\documentclass[leqno]{aadmbook}
\usepackage{amsthm}
\usepackage{amsfonts}
\usepackage{amsmath}
\usepackage{url}
\usepackage{color}

\newcommand{\R}{\mathbb{R}}
\newcommand{\N}{\mathbb{N}}

\newcommand{\Z}{\mathbb{Z}}

\newcommand{\C}{\mathbb{C}}

\newtheorem{definition}{\sc Definition}

\newtheorem{theorem}{\sc Theorem}
\newtheorem{proposition}{\sc Proposition}
\newtheorem{remark}{\sc Remark}
\newtheorem{lemma}{\sc Lemma}
\newtheorem{corollary}{\sc Corollary}

\newcommand{\qhypergeom}[5]{\mbox{$
_#1 \phi_#2\left. \left(\!\!\!\!\!
\begin{array}{c}
\multicolumn{1}{c}{\begin{array}{c} #3
\end{array}}\\[1mm]
\multicolumn{1}{c}{\begin{array}{c} #4
           \end{array}}\end{array}\!\!\!\!
\right| \displaystyle{#5}\right) $} }

\newcommand{\qhypergeomu}[5]{\mbox{$
{_#1 \phi_#2^{\mu}}\left. \left(\!\!\!\!\!
\begin{array}{c}
\multicolumn{1}{c}{\begin{array}{c} #3
\end{array}}\\[1mm]
\multicolumn{1}{c}{\begin{array}{c} #4
           \end{array}}\end{array}\!\!\!\!
\right| \displaystyle{#5}\right) $} }

\newcommand{\qbinomial}[3]{\mbox{$
\biggl[\!\!\!
\begin{array}{c}
#1\\
 #2
\end{array}\!\!\!\biggr]_{
\!{#3}} $} }

\begin{document}


\oddsidemargin 16.5mm
\evensidemargin 16.5mm

\thispagestyle{plain}

%
%

\vspace{5cc}
\begin{center}

{\large\bf  ON FRACTIONAL $q$-EXTENSIONS OF SOME $q$-ORTHOGONAL POLYNOMIALS
\rule{0mm}{6mm}\renewcommand{\thefootnote}{}
\footnotetext{

\rule{2.4mm}{0mm}Keywords and Phrases. fractional calculus, $q$-hypergeometric functions, fractional difference equations.
}}

\vspace{1cc}
{\large\it P. Njionou Sadjang and S. Mboutngam}

\vspace{1cc}
\parbox{24cc}{{\small

Fractional
$q$-extensions of some classical $q$-orthogonal polynomials are introduced and some of the main properties of the new defined functions are given. Next, a fractional $q$-difference equation of Gauss type is introduced and solved by means of power series method.
%

}}
\end{center}

\vspace{1cc}

\vspace{1.5cc}
\begin{center}
{\bf 1. INTRODUCTION}
\end{center}

\noindent Fractional calculus is the field of mathematical analysis which deals with the investigation and applications of derivatives and integrals of arbitrary (real or complex) order. It is  an interesting topic having interconnections with various problems of function theory, integral and differential equations, and other branches of analysis. It has been continually developped, sitmulated by ideas and results in various fields of mathematical analysis. This is demonstrated by the many pubilcations--hundreds of papers in the past years--and by the many conferences devoted to the problems of fractional calculus.

A family $\{P_n(x)\}$, $(n\in\mathbb{N}:=\{0,1,2\dots\},\,k_n\neq
0)$ of polynomials of degree exactly $n$ is a family of classical
$q$-orthogonal polynomials of the $q$-Hahn class if it is the
solution of a $q$-differential equation of the type
\begin{equation}\label{eq2}
    \sigma(x)D_qD_{1/q}P_n(x)+\tau(x)D_qP_n(x)+\lambda_n
    P_n(x)=0,
\end{equation}
where $\sigma(x)=ax^2+bx+c$ is a polynomial of at most second
order and $\tau(x)=dx+e$ is a polynomial of first order. Here, The
$q$-difference operator $D_q$ is defined by
\[D_qf(x)=\frac{f(x)-f(qx)}{(1-q)x},\quad x\neq 0,\,q\neq 1\]
and $D_qf(0):=f'(0)$ by continuity, provided $f'(0)$ exists.

The polynomial systems that are solution of (\ref{eq2}) form the $q$-Hahn tableau.
These systems are contained in the so-called Askey-Wilson scheme
  \cite{KLS}. The following systems are members of the $q$-Hahn tableau:
 the Big $q$-Jacobi polynomials, 
%
  the $q$-Hahn polynomials, 
%
  the Big $q$-Laguerre polynomials, 
%
  the Little $q$-Jacobi polynomials, 
%
 the  $q$-Meixner polynomials, 
%
the Quantum $q$-Krawtchouk polynomials, 
the  $q$-Krawtchouk polynomials, 
 the Affine $q$-Krawtchouk polynomials, 
 the Little $q$-Laguerre polynomials, 
 the $q$-Laguerre polynomials, 
 the Alternative $q$-Charlier (also called $q$-Bessel) polynomials, 
the $q$-Charlier polynomials, 
the Al Salam-Carlitz I polynomials, 
 the Al Salam-Carlitz II polynomials, 
 the Stieltjes-Wigert polynomials,  
the Discrete $q$-Hermite I  polynomials, 
and the Discrete $q$-Hermite II  polynomials. 

In \cite{elena}, the authors have defined some fractional extensions of the Jacobi polynomials from their Rodrigues representation and provided several properties of these new functions. Also, they introduced a fractional version of the Gauss hypergeometric differential equation and used the modified power series method to provide some of its solutions.

In \cite{ishteva}, the authors defined the $C$-Laguerre functions from the Rodrigues representation of the Laguerre polynomials by replacing the ordinary derivative by a fractional type derivative, then they gave several properties of the new defined functions.

In this work, we introduce a new fractional q-differential operator and following previous works (see \cite{elena, ishteva,njionoumboutngam1}), we introduce fractional $q$-extensions of some $q$-orthogonal polynomials of the $q$-Hahn class. The hypergeometric representations of the new defined functions are given and in some cases the limit transitions are provided. 

The paper is organised as follows:
\begin{enumerate}
    \item In Section 2, we present the preliminary results and definitions that are useful for a better reading of this manuscript.
    \item In Section 3, we introduce the fractional $q$-calculus
    \item In Section 4, we introduce a new fractional $q$-differential equation $D_{q^{-1}}^\alpha$ and apply it to some functions,
    \item In Section 5, fractional $q$-extension of some $q$-orthogonal polynomials are given and their basic hypergeometric representation provided. We prove for some of these new defined functions some limit transitions. 
     \item In Section 6, we introduce a fractional $q$-extension of the $q$-hypergeometric
     $q$-difference equation and provide some of its solution.
     are given.
\end{enumerate}

\vspace{1.5cc}
\begin{center}
{\bf 2. PRELIMINARY DEFINITIONS AND RESULTS}
\end{center}

\noindent This section contains some preliminary definitions and results that are useful for a better reading of the manuscript. The $q$-hypergeometric series, a fractional $q$-derivative and fractional $q$-integral are defined. The reader will consult the references \cite{KLS,rainville} for more informations about these concepts.

\begin{definition}
 The basic hypergeometric or $q$-hypergeometric series $_r\phi_s$ is defined by the series
 \[
 \qhypergeom{r}{s}{a_1,\cdots,a_r}{b_1,\cdots,b_s}{q;z}:=\sum_{n=0}^\infty\frac{(a_1,\cdots,a_r;q)_n}{(b_1,\cdots,b_s;q)_n}\left((-1)^n
 q^{\binom{k}{2}}\right)^{1+s-r}\frac{z^n}{(q;q)_n},
 \]
  where
 \[(a_1,\cdots,a_r)_n:=(a_1;q)_n\cdots(a_r;q)_n,\]  with  \[ (a_i;q)_n=\left\{
                                                                          \begin{array}{ll}
                                                                            \prod\limits_{j=0}^{n-1}(1-a_iq^j)&
                                                                             \text{ if }\ n=1,2,3,\cdots \\
                                                                            1 &  \text{ if }\  n=0                                                                          \end{array}
                                                                        \right..
 \]
 For $n=\infty$ we set
\[(a;q)_{\infty}=\prod_{n=0}^{\infty}(1-aq^n),\,\,|q|<1.\]
 The notation $(a;q)_n$ is the so-called $q$-Pochhammer symbol.
 \end{definition}

 \begin{proposition}\cite[Page 16]{KLS}
 The basic hypergeometric series fulfil the following identities
 \begin{eqnarray}
       \qhypergeom{1}{0}{a}{-}{q;z}&=&\sum_{n=0}^{\infty}\frac{(a;q)_n}{(q;q)_n}z^n=\frac{(az;q)_{\infty}}{(z;q)_\infty},\quad 0<|q|<1,\;\; |z|<1,\label{qbinom}\\
       \qhypergeom{1}{1}{a}{c}{\frac{c}{a}}&=&\frac{(c/a;q)_{\infty}}{(c;q)_{\infty}},\quad 0<|q|<1.
 \end{eqnarray}
 \end{proposition}

 \noindent Relation \eqref{qbinom} is the so-called $q$-binomial theorem.

The next proposition gives some important Heine transformation formulas for basic hypergeometric series.
 \begin{proposition}
 The following transformation formulas hold (\cite[P. 10]{Gasper-Rahman})
 \begin{eqnarray}
     \qhypergeom{2}{1}{a,b}{c}{z}&=&\dfrac{(az,b;q)_{\infty}}{(c,z;q)_{\infty}}\qhypergeom{2}{1}{c/b,z}{az}{q;b}, \quad |z|<1,\label{heine}\\
     \qhypergeom{2}{1}{a,b}{c}{q;z}&=&\dfrac{(abz/c;q)_{\infty}}{(z;q)_{\infty}}\qhypergeom{2}{1}{c/a,b/a}{c}{q;abz/c}\label{heine2}\\
     \qhypergeom{2}{1}{a,b}{c}{z}&=&\dfrac{(az;q)_{\infty}}{(z;q)_{\infty}}\qhypergeom{2}{2}{a,c/b}{c,az}{bz}, \quad |z|<1,\nonumber \\
     \qhypergeom{2}{1}{a,b}{0}{z}&=& \dfrac{(az;q)_{\infty}}{(z;q)_{\infty}}\qhypergeom{1}{1}{a}{az}{q;bz},\quad |z|<1.\label{heine3}
 \end{eqnarray}
 \end{proposition}

\begin{proposition}\cite{koorwinder}
The $q$-hypergeometric $q$-difference equation
\begin{eqnarray}
&&z(q^c-q^{a+b+1}z)(D^2_q u)(z)\label{korn1}\\
&&\;\; +\left([c]_q-(q^b[a]_q+q^a[b+1]_q)z\right)(D_qu)(z)-[a]_q[b]_qu(z)=0\nonumber
\end{eqnarray}
has particular solutions
\[u_1(z)=\qhypergeom{2}{1}{q^a,q^b}{q^c}{q;z},\quad\textrm{and}\quad u_2(z)=z^{1-c}\qhypergeom{2}{1}{q^{1+a-c},q^{1+b-c}}{q^{2-c}}{q;z}. \]
\end{proposition}

 From the definition of $(a;q)_{\infty}$, it follows that for $0<|q|<1$, and for a nonnegative integer $n$, we have
 \[(a;q)_n=\dfrac{(a;q)_{\infty}}{(aq^n;q)_{\infty}}.\]

 \begin{definition}\label{def2}
 For any complex number $\lambda$,
 \[(a;q)_\lambda=\dfrac{(a;q)_{\infty}}{(aq^\lambda;q)_{\infty}},\quad 0<|q|<1,\]
 where the principal value of $q^\lambda$ is taken.
 \end{definition}

 We will also use the following common notations
\begin{equation*}\label{qbraquets}\index{$q$-brackets}
[a]_q=\dfrac{1-q^a}{1-q},\quad a\in\C,\quad \textcolor{black}{ q\neq 1},
\end{equation*}
\begin{equation*}\index{$q$-binomial}
\qbinomial{n}{m}{q}=\dfrac{(q;q)_n}{(q;q)_m(q;q)_{n-m}},\quad \textcolor{black}{ 0\leq m\leq n},
\end{equation*}
called the $q$-bracket and the $q$-binomial coefficients, respectively.

%
%

\begin{definition}[see \cite{kac}]
Suppose $0<a<b$. The definite $q$-integral is defined as
\begin{equation}
\int_0^bf(x)d_qx=(1-q)b\sum_{n=0}^{\infty}q^nf(q^nb),
\end{equation}
and
\begin{equation}
\int_a^bf(x)d_qx=\int_0^bf(x)d_qx-\int_0^af(x)d_qx.
\end{equation}
\end{definition}

\begin{definition}
The $q$-Gamma function is defined by
\begin{equation}
\Gamma_q(x):=\dfrac{(q;q)_{\infty}}{(q^x;q)_{\infty}}(1-q)^{1-x},\;\; 0<q<1.
\end{equation}
\end{definition}

\begin{remark}
From Definition \ref{def2}, the $q$-Gamma function is also represented by
\[\Gamma_q(x)=(1-q)^{1-x}(q;q)_{x-1}.\]
Note also that the $q$-Gamma function satisfies the functional equation
\[\Gamma_q(x+1)=[x]_q\Gamma_q(x),\quad \textrm{with}\quad \Gamma_q(1)=1.\]
\end{remark}

\noindent Note that for arbitrary complex $\alpha$,
\begin{equation}\label{q-binom-alpha}
 \qbinomial{\alpha}{k}{q}=\dfrac{(q^{-\alpha};q)_{k}}{(q;q)_k}(-1)^kq^{\alpha k-\binom{k}{2}}=\dfrac{\Gamma_q(\alpha+1)}{\Gamma_q(k+1)\Gamma_q(\alpha-1)}.
 \end{equation}

\noindent The exponential function has two different natural $q$-extensions, denoted by $e_q(z)$ and $E_q(z)$, which can be defined by
\begin{equation}
e_q(z):=\qhypergeom{1}{0}{0}{-}{q;z}=\sum_{n=0}^{\infty}\frac{z^n}{(q;q)_n}=\frac{1}{(z;q)_{\infty}},\quad 0<|q|<1,\;\; |z|<1, \label{small-qexp}
\end{equation}
and
\begin{equation}
E_q(z):=\qhypergeom{0}{0}{-}{-}{q,-z}=\sum_{n=0}^{\infty}\frac{q^{\binom{n}{2}}}{(q;q)_n}z^n=(-z;q)_\infty,\quad 0<|q|<1.  \label{big-qexp}
\end{equation}
These $q$-analogues of the exponential function are clearly related by
\[e_q(z)E_q(-z)=1.\]

\begin{theorem}
The following expansions hold true.
\begin{eqnarray}
  e_q(\alpha z)e_q(\beta z)&=&\sum_{n=0}^{\infty} \left[\sum_{k=0}^n\qbinomial{n}{k}{q}\left(\dfrac{\alpha}{\beta}\right)^k\right]\dfrac{(\beta z)^n}{(q;q)_n};  \\
  e_q(\alpha z)E_q(\beta z)&=& \sum_{n=0}^{\infty}\dfrac{(-\beta \alpha^{-1};q)_n}{(q;q)_n}(\alpha z)^n=\qhypergeom{1}{0}{-\beta\alpha^{-1}}{-}{q;\alpha z};\label{gen-qbinom}  \\
  E_q(\alpha z)E_q(\beta z)&=& \sum_{n=0}^{\infty} \left[\sum_{k=0}^n\qbinomial{n}{k}{q}\left(\dfrac{\alpha}{\beta}\right)^k q^{\binom{k}{2}+\binom{n-k}{2}}\right]\dfrac{(\beta z)^n}{(q;q)_n}.
\end{eqnarray}
\end{theorem}

\begin{proof}
From the expansion of $e_q(z)$, we have respectively
\[e_q(\alpha z)=\sum_{n=0}^{\infty}\frac{\alpha^nz^n}{(q;q)_n}\quad \textrm{and} \quad e_q(\beta z)=\sum_{n=0}^{\infty}\frac{\beta^nz^n}{(q;q)_n}.\]
It follows that
\begin{eqnarray*}
e_q(\alpha z)e_q(\beta z)&=&\sum_{n=0}^{\infty}\frac{\alpha^nz^n}{(q;q)_n}\times \sum_{n=0}^{\infty}\frac{\beta^nz^n}{(q;q)_n}\\
&=& \sum_{n=0}^{\infty} \sum_{k=0}^n\frac{\alpha^kz^k}{(q;q)_k}\frac{\beta^{n-k}z^{n-k}}{(q;q)_{n-k}}\\
&=& \sum_{n=0}^{\infty} \left[\sum_{k=0}^n\frac{(q;q)_n\alpha^k\beta^{-k}}{(q;q)_k(q;q)_{n-k}}\right]\dfrac{(\beta z)^n}{(q;q)_n}\\
&=& \sum_{n=0}^{\infty} \left[\sum_{k=0}^n\qbinomial{n}{k}{q}\left(\dfrac{\alpha}{\beta}\right)^k\right]\dfrac{(\beta z)^n}{(q;q)_n}.
\end{eqnarray*}
Next,
\begin{eqnarray*}
e_q(\alpha z)E_q(\beta z)&=&\sum_{n=0}^{\infty}\frac{\alpha^nz^n}{(q;q)_n}\times \sum_{n=0}^{\infty}\frac{q^{\binom{n}{2}}}{(q;q)_n}\beta^nz^n    \\
&=& \sum_{n=0}^{\infty}\sum_{k=0}^n\dfrac{\alpha^k z^k}{(q;q)_k}\dfrac{q^{\binom{n-k}{2}}\beta^{n-k}z^{n-k}}{(q;q)_{n-k}}\\
&=& \sum_{n=0}^{\infty} \left[\sum_{k=0}^n\qbinomial{n}{k}{q}\left(\dfrac{\alpha}{\beta}\right)^k q^{\binom{n-k}{2}}\right]\dfrac{(\beta z)^n}{(q;q)_n}\\
&=& \sum_{n=0}^{\infty}\dfrac{(-\beta \alpha^{-1};q)_n}{(q;q)_n}(\alpha z)^n=\qhypergeom{1}{0}{-\beta\alpha^{-1}}{-}{q;\alpha z}.
\end{eqnarray*}
Note that if we take $\alpha=1$ and $\beta=-a$ in \eqref{gen-qbinom}, we obtain \eqref{qbinom}. So,  \eqref{gen-qbinom} is a generalization of \eqref{qbinom}. \\ Finally,
\begin{eqnarray*}
E_q(\alpha z)E_q(\beta z)&=& \sum_{n=0}^{\infty}\sum_{k=0}^n\dfrac{\alpha^k q^{\binom{k}{2}} z^k}{(q;q)_k}\dfrac{q^{\binom{n-k}{2}}\beta^{n-k}z^{n-k}}{(q;q)_{n-k}}\\
&=& \sum_{n=0}^{\infty} \left[\sum_{k=0}^n\qbinomial{n}{k}{q}\left(\dfrac{\alpha}{\beta}\right)^k q^{\binom{k}{2}+\binom{n-k}{2}}\right]\dfrac{(\beta z)^n}{(q;q)_n}.
\end{eqnarray*}
\end{proof}

\begin{corollary}
The following expansions apply.
\begin{eqnarray*}
\dfrac{e_q(\alpha z)}{e_q(\beta z)}&=& \sum_{n=0}^{\infty}\dfrac{(\beta \alpha^{-1};q)_n}{(q;q)_n}(\alpha z)^n=\qhypergeom{1}{0}{\beta\alpha^{-1}}{-}{q;\alpha z};     \\
\dfrac{e_q(\alpha z)}{E_q(\beta z)}&=&  \sum_{n=0}^{\infty} \left[\sum_{k=0}^n\qbinomial{n}{k}{q}\left(-\dfrac{\alpha}{\beta}\right)^k\right]\dfrac{(-\beta z)^n}{(q;q)_n}.
\end{eqnarray*}
\end{corollary}

\begin{proof}
For the first relation we use the equation $ \dfrac{e_q(\alpha z)}{e_q(\beta z)}=e_q(\alpha z)E_q(-\beta z)$ and for the second relation we use $ \dfrac{e_q(\alpha z)}{E_q(\beta z)}=e_q(\alpha z)e_q(-\beta z)$.
\end{proof}

\begin{proposition}
The following relations are valid.
\begin{eqnarray*}
\qhypergeom{1}{0}{a}{-}{q;z}\qhypergeom{1}{0}{b}{-}{q;z}&=& \sum_{n=0}^{\infty}\left[\sum_{k=0}^n\qbinomial{n}{k}{q}(a;q)_k(b;q)_{n-k}\right]\dfrac{z^n}{(q;q)_n};\\
\qhypergeom{1}{0}{a}{-}{q;bz}\qhypergeom{1}{0}{b}{-}{q;z}&=&\qhypergeom{1}{0}{ab}{-}{q;z}.
\end{eqnarray*}
\end{proposition}

\vspace{1.5cc}
\begin{center}
{\bf 3. A BRIEF REVIEW OF FRACTIONAL $q$-CALCULUS}
\end{center}

The usual starting point for a definition of fractional operators in $q$-calculus taken in \cite{annaby,salam, argawal,rajkovic,rajkovic2}, is the $q$-analogue of the Riemann-Liouville fractional integral
\begin{equation}\label{fract-q-int1}
I_q^{\alpha}f(z)=\frac{z^{\alpha-1}}{\Gamma_q(\alpha)}\int_0^z(tq/z;q)_{\alpha-1}f(t)d_qt
=\frac{1}{\Gamma_q(\alpha)}\int_0^z(z\ominus qt)_q^{\alpha-1}f(t)d_qt.
\end{equation}
This $q$-integral was motivated from the $q$-analogue of the Cauchy formula for a repeated $q$-integral
\begin{eqnarray}
I_{q,a}^nf(z)&=&\int_a^zd_qt\int_a^t d_{q}t_{n-1} \int_{a}^{t_{n-1}}d_{q}t_{n-2}\cdots\int_{a}^{t_2}f(t_1)d_qt_1\\
\nonumber &=& \dfrac{z^{n-1}}{[n-1]_q!]}\int_{0}^z(tq/z;q)_{n-1}f(t)d_qt.
\end{eqnarray}
The reduction of the multiple $q$-integral to a single one was considered by Al-Salam in \cite{salam2}.
 In \cite{rajkovic2}, the authors allow the lower parameter in \eqref{fract-q-int1} to be
 any real number $a\in (0,z)$. There are several definitions of the fractional $q$-integral and fractional
  $q$-derivatives.
We adopt in this work the definition of the fractional $q$-integral given in \cite{rajkovic}.

\begin{definition}\cite{rajkovic}
The fractional $q$-integral is
\[ (I_{q,c}^\alpha f)(x)=\dfrac{x^{\alpha-1}}{\Gamma_q(\alpha)}\int_{c}^{x}
(qt/x;q)_{\alpha-1}f(t)d_qt=\dfrac{1}{\Gamma_q(\alpha)}\int_c^x(x\ominus
qt)_{q}^{\alpha-1} f(t)d_qt,\;\; (\alpha\in\R^+).\]
\end{definition}

\noindent The following lemma is of great importance in the sequel.

\begin{lemma}\cite{rajkovic}
For $\alpha\in\R^+$, $\lambda,\lambda+\alpha\in \R\setminus\{-1,-2,\cdots\}$, the following fractional $q$-integral is valid:
\begin{equation}
I_{q,c}^{\alpha}(x\ominus c)^{\lambda}_{q}=\dfrac{\Gamma_q(\lambda+1)}{\Gamma_q(\alpha+\lambda+1)}
(x\ominus c)^{\alpha+\lambda}_{q},\quad (0<c<x).
\end{equation}
\end{lemma}

\begin{definition}\cite{rajkovic}
The fractional $q$-derivative of Riemann-Liouville type of order $\alpha\in\R^+$ is
\begin{equation}
(D_{q,c}^\alpha f)(x)=(D_{q}^{\lceil \alpha \rceil}I_{q,c}^{\lceil \alpha\rceil-\alpha}f)(x),
\end{equation}
where $\lceil \alpha\rceil$ denotes the smallest integer greater or equal to $\alpha$.
\end{definition}
\noindent Mahmoud Annaby and Zeinab Mansour {\cite[P.\ 148]{annaby}} prove
that the Riemann Liouville fractional operator $D_{q,0}^\alpha$
coincides with a $q$-analogue of the Gr\"{u}nwald Letnikov
fractional operator defined by
\begin{eqnarray}
(D_{q}^\alpha
f)(x)&=&\dfrac{1}{x^{\alpha}(1-q)^{\alpha}}\sum_{n=0}^{\infty}(-1)^n\qbinomial{\alpha}{n}{q}\dfrac{f(q^nx)}{q^{\frac{n(n-1)}{2}+n(\alpha-n)}}\nonumber\\
&=&\dfrac{1}{x^{\alpha}(1-q)^{\alpha}}\sum_{n=0}^{\infty}\dfrac{(q^{-\alpha};q)_n}{(q;q)_n}q^nf(q^nx).
\label{zeinab}
\end{eqnarray}
In the sequel of this work, we will most of the time make use of this definition of the
fractional $q$-derivative of order $\alpha$.
\begin{theorem}\cite{rajkovic}
Let $\alpha, \beta \in \R_+$. The fractional $q$-integral has the
following property \[
 (I_{q,c}^\alpha I_{q,c}^\beta f)(x)=(I_{q,c}^{\alpha +\beta}
 f)(x),\;\;0<c<x.
\]
\end{theorem}

\begin{theorem}\cite{rajkovic}
For $0<c<x$,   the operators $I_{q,c}^\alpha$ and $D_{q,c}^\alpha$
satisfy the following property
\[
 (D_{q,c}^\alpha I_{q,c}^\alpha f)(x)=f(x).
\]
\end{theorem}
\noindent The following lemma is of great importance in the sequel.
\begin{lemma}\cite{rajkovic}
For $\lambda\in\R\setminus\{-1,-2,\cdots\}$ and $0<c<x$, the following relation is valid:
\begin{equation}
D_{q,c}^{\alpha}(x\ominus c)^{\lambda}_{q}=\left\{\begin{array}{cl}
\dfrac{\Gamma_q(\lambda+1)}{\Gamma_q(\lambda-\alpha+1)}(x\ominus c)^{\lambda-\alpha}_q,& \alpha-\lambda\in \R\setminus \N,\\
0, & \alpha-\lambda\in\N.
\end{array}\right.
\end{equation}
\end{lemma}

\vspace{1.5cc}
\begin{center}
{\bf 4. MORE FRACTIONAL $q$-OPERATOR}
\end{center}

Since many Rodrigues-type formulas for some of the orthogonal polynomials of the $q$-Hahn class are expressed in terms of the $q$-operator $D_{q^{-1}}$ instead of $D_{q}$, and since our new functions are defined using the Rodrigues-type formula of each family, there is a need to develop a fractional calculus for $D_{q^{-1}}$. The more natural way to do it is to start by the power derivative of $D_{q^{-1}}$. The following proposition (see \cite{Kornelia}) gives the result.

\begin{proposition}
Let $n\in \N_{0}$ and $f$ a given function defined on $\{q^k,\; k\in\Z\}$. Then the following power derivative rule for $D_{q^{-1}}$ applies
\begin{equation}\label{dq0}
  D^{n}_{q^{-1}} f(x)=\dfrac{q^{\binom{n+1}{2}}}{(1-q)^nx^n}\sum_{k=0}^n (-1)^k\qbinomial{n}{k}{q} q^{\binom{k}{2}-(n-1)k}f(q^{k-n}x).
\end{equation}
\end{proposition}

\begin{proof}
This proof is also from \cite{Kornelia}. The result is clear for $n=0$. Assume the assertion is true for $n\geq 0$, then:
\begin{eqnarray*}
D_{q^{-1}}^{n+1}f(x)&=&\dfrac{q^{\binom{n+1}{2}}}{(1-q)^n}D_{q^{-1}}\left(x^{-n}\sum_{k=0}^n (-1)^k\qbinomial{n}{k}{q} q^{\binom{k}{2}-(n-1)k}f(q^{k-n}x)\right)\\
&=&\dfrac{q^{\binom{n+1}{2}}}{(1-q)^n}\dfrac{q}{(1-q)x}\left( x^{-n}\sum_{k=0}^n(-1)^k\qbinomial{n}{k}{q}q^{\binom{k}{2}-(n-1)k}f(q^{k-n}x)\right.\\
&&\left.-q^nx^{-n}\sum_{k=0}^n(-1)^k\qbinomial{n}{k}{q}q^{\binom{k}{2}-(n-1)k}f(q^{k-n-1}x)\right)\\
&=&\dfrac{-q^{\binom{n+1}{2}}q^{n+1}}{(1-q)^{n+1}x^{n+1}}\left(\sum_{k=1}^{n+1}(-1)^{k+1}\qbinomial{n}{k-1}{q}q^{\binom{k-1}{2}-(n-1)(k-1)-n}f(q^{k-n-1}x)\right.\\
&&\left.-\sum_{k=0}^n(-1)^k\qbinomial{n}{k}{q}q^{\binom{k}{2}-(n-1)k}f(q^{k-n-1}x)\right)\\
&=&\dfrac{q^{\binom{n+2}{2}}}{(1-q)^{n+1}x^{n+1}}\left(\sum_{k=1}^{n+1}(-1)^{k}\qbinomial{n}{k-1}{q}q^{\binom{k-1}{2}-(n-1)(k-1)-n}f(q^{k-n-1}x)\right.\\
&&\left.-\sum_{k=0}^n(-1)^k\qbinomial{n}{k}{q}q^{\binom{k}{2}-(n-1)k}f(q^{k-n-1}x)\right)\\
&=&\dfrac{q^{\binom{n+2}{2}}}{(1-q)^{n+1}x^{n+1}}\sum_{k=0}^n(-1)^k\qbinomial{n+1}{k}{q}q^{\binom{k}{2}-nk}f(q^{k-n-1}x).
\end{eqnarray*}
\end{proof}

\noindent Note that, using the obvious relation $\qbinomial{n}{k}{q}=\qbinomial{n}{n-k}{q}$, and reversing the order of summation, \eqref{dq0} reads
\begin{eqnarray*}
  D^{n}_{q^{-1}} f(x)&=&\dfrac{q^{\binom{n+1}{2}}}{(1-q)^nx^n}\sum_{k=0}^n (-1)^{n-k}\qbinomial{n}{n-k}{q} q^{\binom{n-k}{2}-(n-1)(n-k)}f(xq^{-k})\\
  &=&\dfrac{q^{\binom{n+1}{2}}}{(q-1)^nx^n}\sum_{k=0}^n (-1)^{k}\qbinomial{n}{k}{q} q^{\binom{n-k}{2}-(n-1)(n-k)}f(xq^{-k}).
\end{eqnarray*}
Next, using the fact that
\[  \binom{n-k}{2}-(n-1)(n-k)=\dfrac{(n-k)(n-k-1)-2(n-1)(n-k)}{2}=-\binom{n}{2}+\binom{k}{2},\]
we get
\begin{eqnarray}
  D^{n}_{q^{-1}} f(x)&=&\dfrac{q^{\binom{n+1}{2}}}{(q-1)^nx^n}\sum_{k=0}^n (-1)^{k}\qbinomial{n}{k}{q} q^{-\binom{n}{2}+\binom{k}{2}}f(xq^{-k})\nonumber\\
  &=&\dfrac{q^{n}}{(q-1)^nx^n}\sum_{k=0}^n (-1)^{k}\qbinomial{n}{k}{q} q^{\binom{k}{2}}f(xq^{-k}).\label{dq3}
\end{eqnarray}

\begin{remark}
It is known that (see \cite{annaby0,Kornelia})
\begin{equation}\label{dq1}
D_q^nf(x)=\dfrac{1}{(1-q)^nx^n}\sum_{k=0}^n(-1)^k\qbinomial{n}{k}{q}q^{\binom{k}{2}-(n-1)k}f(q^kx).
\end{equation}
Note that replacing $q$ by $q^{-1}$ in  \eqref{dq1}, it follows that
\begin{eqnarray*}
D_{q^{-1}}^nf(x)&=&\dfrac{1}{(1-q^{-1})^nx^n}\sum_{k=0}^n(-1)^k\qbinomial{n}{k}{q^{-1}}q^{-\binom{k}{2}-(n-1)k}f(q^{-k}x)\\
&=&\dfrac{q^n}{(q-1)^nx^n}\sum_{k=0}^n(-1)^k\qbinomial{n}{k}{q^{-1}}q^{-\binom{k}{2}-(n-1)k}f(q^{-k}x)
\end{eqnarray*}
Taking care that
\[[n]_{q^{-1}}=\dfrac{1-q^{-n}}{1-q^{-1}}=\dfrac{1}{q^{n-1}}[n]_q,\]
it follows that
\[[n]_{q^{-1}}!=q^{-\binom{n}{2}}[n]_q!,\]
and so we get
\[\qbinomial{n}{k}{q^{-1}}=\dfrac{q^{-\binom{n}{2}}[n]_q!}{q^{-\binom{k}{2}}[k]_q!q^{-\binom{n-k}{2}}[n-k]_q!}=q^{\binom{k}{2}+\binom{n-k}{2}-\binom{n}{2}}\qbinomial{n}{k}{q}.\]
Now, we can write
\begin{eqnarray*}
D_{q^{-1}}^nf(x)&=&\dfrac{1}{(1-q^{-1})^nx^n}\sum_{k=0}^n(-1)^kq^{\binom{k}{2}+\binom{n-k}{2}-\binom{n}{2}}\qbinomial{n}{k}{q}q^{-\binom{k}{2}-(n-1)k}f(q^{-k}x)\\
&=&\dfrac{q^n}{(q-1)^nx^n}\sum_{k=0}^n(-1)^k\qbinomial{n}{k}{q}q^{\binom{n-k}{2}-\binom{n}{2}-(n-1)k}f(q^{-k}x)\\
&=&\dfrac{q^n}{(q-1)^nx^n}\sum_{k=0}^n(-1)^k\qbinomial{n}{k}{q}q^{\binom{k}{2}-2\binom{n}{2}+(n-1)(n-k)-(n-1)k}f(q^{-k}x)\\
&=& \dfrac{q^{n}}{(q-1)^nx^n}\sum_{k=0}^n (-1)^{k}\qbinomial{n}{k}{q} q^{\binom{k}{2}}f(xq^{-k}).
\end{eqnarray*}
This is exactely another way to write the result \eqref{dq0} obtained in \cite{Kornelia} thanks to \eqref{dq3}.
\end{remark}

\noindent We are about to define a fractional extension of $D_{q^{-1}}$.  Since $\displaystyle{\qbinomial{n}{k}{q}=0}$ for $k>n$, we can write \eqref{dq3} as
\begin{equation}\label{dq4}
D_{q^{-1}}^nf(x)= \dfrac{q^{n}}{(q-1)^nx^n}\sum_{k=0}^\infty (-1)^{k}\qbinomial{n}{k}{q} q^{\binom{k}{2}}f(xq^{-k})
\end{equation}
\noindent Equation \eqref{dq4} suggests that for any arbitrary complex number $\alpha$, $D^{\alpha}_{q^{-1}}$ could be defined by
\[ D_{q^{-1}}^{\alpha}f(x)= \dfrac{q^{\alpha}}{(q-1)^\alpha x^\alpha}\sum_{k=0}^\infty (-1)^{k}\qbinomial{\alpha}{k}{q} q^{\binom{k}{2}}f(xq^{-k}),  \]
provided that the infinite series of the right hand side converges.
Now, using equation \eqref{q-binom-alpha}, we obtain
\begin{eqnarray*}
 D_{q^{-1}}^{\alpha}f(x)&=& \dfrac{q^{\alpha}}{(q-1)^\alpha x^\alpha}\sum_{k=0}^\infty (-1)^{k}\dfrac{(q^{-\alpha};q)_{k}}{(q;q)_k}(-1)^kq^{\alpha k-\binom{k}{2}} q^{\binom{k}{2}}f(xq^{-k})\\
 &=&\dfrac{q^{\alpha}}{(q-1)^\alpha x^\alpha}\sum_{k=0}^\infty\dfrac{(q^{-\alpha};q)_{k}}{(q;q)_k}q^{\alpha k}f(xq^{-k})
\end{eqnarray*}
We then set the following definition.
\begin{definition}
For any complex number $\alpha$, we define the fractional operator $D_{q^{-1}}^{\alpha}$ by
\begin{equation}\label{def-dq-1}
D_{q^{-1}}^{\alpha}f(x)=\dfrac{q^{\alpha}}{(q-1)^\alpha x^\alpha}\sum_{k=0}^\infty\dfrac{(q^{-\alpha};q)_{k}}{(q;q)_k}q^{\alpha k}f(xq^{-k}),
\end{equation}
provided that the right hand side of \eqref{def-dq-1} converges.
\end{definition}

\noindent Note also that we could use directly \eqref{dq0} to write
\begin{eqnarray*}
_\star D^{\alpha}_{q^{-1}}f(x)&=& \dfrac{q^{\frac{\alpha(\alpha-1)}{2}}}{(1-q)^\alpha x^\alpha}\sum_{k=0}^{\infty} (-1)^k\qbinomial{\alpha}{k}{q} q^{\binom{k}{2}-(\alpha-1)k}f(q^{k-\alpha}x)\\
&=&\dfrac{q^{\frac{\alpha(\alpha-1)}{2}}}{(1-q)^\alpha x^\alpha}\sum_{k=0}^{\infty} (-1)^k\dfrac{(q^{-\alpha};q)_{k}}{(q;q)_k}(-1)^kq^{\alpha k-\binom{k}{2}} q^{\binom{k}{2}-(\alpha-1)k}f(q^{k-\alpha}x)\\
&=&\dfrac{q^{\frac{\alpha(\alpha-1)}{2}}}{(1-q)^\alpha x^\alpha}\sum_{k=0}^{\infty} \dfrac{(q^{-\alpha};q)_{k}}{(q;q)_k} q^{k}f(q^{k-\alpha}x)
\end{eqnarray*}
which may be looked as another fractional extension of $D_{q^{-1}}$.

\begin{proposition} For $\alpha\in \R\setminus \{1,2,3,\cdots$, 
the following derivative rule applies
\begin{equation}
D_{q^{-1}}^{\lambda}x^{\alpha}= \left\{\begin{array}{rcl}
\dfrac{q^{\lambda}}{(q-1)^\lambda}\dfrac{(q^{-\alpha};q)_{\infty}}{(q^{\lambda-\alpha};q)_{\infty}}x^{\alpha-\lambda},& \textrm{if}&  \alpha-\lambda\in \R\setminus \N \\
0,&  \textrm{if}& \alpha-\lambda \in \N
\end{array} \right..
\end{equation}
\end{proposition}

\begin{proof}
From \eqref{def-dq-1}, we have 
\begin{eqnarray*}
 D_{q^{-1}}^\lambda x^{\alpha}&=& \dfrac{q^{\lambda}}{(q-1)^\lambda x^\lambda}\sum_{k=0}^\infty\dfrac{(q^{-\lambda};q)_{k}}{(q;q)_k}q^{\lambda k}(xq^{-k})^{\alpha}\\
 &=&\dfrac{q^{\lambda}x^{\alpha}}{(q-1)^\lambda x^\lambda}\sum_{k=0}^\infty\dfrac{(q^{-\lambda};q)_{k}}{(q;q)_k}q^{(\lambda-\alpha) k}\\
 &=&\dfrac{q^{\lambda}}{(q-1)^\lambda}             x^{\alpha-\lambda} \qhypergeom{1}{0}{q^{-\lambda}}{-}{q;q^{\lambda-\alpha}}\\
 &=& \dfrac{q^{\lambda}}{(q-1)^\lambda}\dfrac{(q^{-\alpha};q)_{\infty}}{(q^{\lambda-\alpha};q)_{\infty}}x^{\alpha-\lambda}.
\end{eqnarray*}
\end{proof}

\vspace{1.5cc}
\begin{center}
{\bf 5. FRACTIONAL $q$-EXTENSIONS OF SOME $q$-ORTHOGONAL POLYNOMIALS}
\end{center}

\noindent In this section we introduce some fractional $q$-extensions of some orthogonal polynomials of the $q$-Hahn class. The families that are of interest here are those which use $D_q$ and $D_{q^{-1}}$ in their Rodrigues representations.

\begin{flushleft}
{\bf 5.1. The fractional Big $q$-Jacobi functions}
\end{flushleft}

\noindent The Big $q$-Jacobi polynomials have the $q$-hypergeometric representation \cite[P.\ 438]{KLS}
 \[p_n(x;a,b,c;q)=\qhypergeom{3}{2}{q^{-n},abq^{n+1},x}{aq,cq}{q;q}.\]
They can also be represented by the Rodrigues-type formula \cite[P.\ 438]{KLS}
\[w(x;a,b,c;q)P_n(x;a,b,c;q)=\dfrac{a^nc^nq^{n(n+1)}(1-q)^n}{(aq,cq;q)_n}D_q^n[w(x;aq^n,bq^n,cq^n;q)],\]
with
\[w(x;a,b,c;q)=\dfrac{(a^{-1}x,c^{-1}x;q)_{\infty}}{(x,bc^{-1}x;q)_{\infty}}.\]

\begin{definition}
Let $\lambda\in\R$, we define the fractional Big $q$-Jacobi
functions by
\begin{equation}\label{big q-Jacobi}
P_{\lambda}(x;a,b,c;q)=\dfrac{a^\lambda c^\lambda
q^{\lambda(\lambda+1)}(1-q)^\lambda\left(x,{b\over
c}x;q\right)_{\infty}}{(aq,cq;q)_\lambda \left({x\over a},{x\over
c};q\right)_{\infty}} D_q^{\lambda}\left[\dfrac{\left({x\over
aq^{\lambda}},{x\over
cq^{\lambda}};q\right)_{\infty}}{\left(x,{b\over
c}x;q\right)_{\infty}} \right].
\end{equation}
\end{definition}

\begin{proposition}
The fractional Big $q$-Jacobi functions defined by relation
(\ref{big q-Jacobi}) have the following basic hypergeometric
representation
\begin{eqnarray*}
P_{\lambda}(x;a,b,c;q)=\dfrac{a^\lambda c^\lambda
q^{\lambda(\lambda+1)}}{(aq,cq;q)_\lambda \left({x\over a},{x\over
c};q\right)_{-\lambda}} \qhypergeom{3}{2}{q^{-\lambda},x,{b\over
c}x}{{x\over a}q^{-\lambda},{x\over c}q^{-\lambda}}{q;q}.
\end{eqnarray*}
\end{proposition}

\begin{proof}
Using the definitions of the fractional Big $q$-Jacobi functions and
 the fractional $q$-derivative (\ref{zeinab}), we have:
\begin{eqnarray*}
P_{\lambda}(x;a,b,c;q)&=&\dfrac{a^\lambda c^\lambda
q^{\lambda(\lambda+1)}(1-q)^\lambda\left(x,{b\over
c}x;q\right)_{\infty}}{(aq,cq;q)_\lambda \left({x\over a},{x\over
c};q\right)_{\infty}} D_q^{\lambda}\left[\dfrac{\left({x\over
aq^{\lambda}},{x\over
cq^{\lambda}};q\right)_{\infty}}{\left(x,{b\over
c}x;q\right)_{\infty}} \right]\\
&=&\dfrac{a^\lambda c^\lambda q^{\lambda(\lambda+1)}\left(x,{b\over
c}x;q\right)_{\infty}x^{-\lambda}}{(aq,cq;q)_\lambda \left({x\over
a},{x\over c};q\right)_{\infty}}
\sum_{n=0}^{+\infty}q^n\dfrac{(q^{-\lambda};q)_{n}}{(q;q)_{n}}
\dfrac{\left({xq^{n-\lambda}\over a},{xq^{n-\lambda}\over
c};q\right)_{\infty}}{\left(xq^{n},{b\over
c}xq^{n};q\right)_{\infty}}.
\end{eqnarray*}
Using the properties of the $q$-pochhammer symbol, we have:
\begin{eqnarray*}
P_{\lambda}(x;a,b,c;q)&=&\dfrac{a^\lambda c^\lambda
q^{\lambda(\lambda+1)}\left(x,{b\over
c}x;q\right)_{\infty}x^{-\lambda}}{(aq,cq;q)_\lambda \left({x\over
a},{x\over c};q\right)_{\infty}}\\
&&\times\sum_{n=0}^{+\infty}q^n\dfrac{(q^{-\lambda};q)_{n}}{(q;q)_{n}}
\dfrac{\left({xq^{-\lambda}\over a},{xq^{-\lambda}\over
c};q\right)_{\infty}}{\left({xq^{-\lambda}\over
a},{xq^{-\lambda}\over c};q\right)_{n}} \dfrac{\left(x,{xb\over
c};q\right)_{n}}{\left(x,{xb\over c};q\right)_{\infty}}.
\end{eqnarray*}
Using the fact that
$(aq^\lambda;q)_\infty=\dfrac{(a;q)_\infty}{(a;q)_\lambda}$ then, we
have
\begin{eqnarray*}
P_{\lambda}(x;a,b,c;q)&=&\dfrac{a^\lambda c^\lambda
q^{\lambda(\lambda+1)}x^{-\lambda}}{(aq,cq;q)_\lambda \left({x\over
a},{x\over c};q\right)_{-\lambda}}
\sum_{n=0}^{+\infty}q^n\dfrac{\left(q^{-\lambda},x,{b\over
c}x;q\right)_n}{\left({xq^{-\lambda}\over a},{xq^{-\lambda}\over
c};q\right)_{n}}\dfrac{q^n}{(q;q)_n}\\
&=&\dfrac{a^\lambda c^\lambda
q^{\lambda(\lambda+1)}}{(aq,cq;q)_\lambda \left({x\over a},{x\over
c};q\right)_{-\lambda}} \qhypergeom{3}{2}{q^{-\lambda},x,{b\over
c}x}{{x\over a}q^{-\lambda},{x\over c}q^{-\lambda}}{q;q}.
\end{eqnarray*}
\end{proof}

\begin{flushleft}
{\bf 5.3. The fractional Big $q$-Laguerre functions}
\end{flushleft}

\noindent The Big $q$-Laguerre polynomials have the $q$-hypergeometric representation \cite[P.\ 478]{KLS}
\[P_n(x,a,b;q)=\qhypergeom{3}{2}{q^{-n},0,x}{aq,bq}{q;q}.\]
They can also be represented by the Rodrigues-type formula  \cite[P.\ 480]{KLS}
\[w(x;a,b;q)P_n(x;a,b;q)=\dfrac{a^nb^nq^{n(n+1)}(1-q)^n}{(aq,bq;q)_n}D_q^n[w(x;aq^n,bq^n;q)]\]
where
\[w(x;a,b;q)=\dfrac{(a^{-1}x,b^{-1}x;q)_{\infty}}{(x;q)_{\infty}}.\]

\begin{definition}
Let $\lambda\in\R$, we define the fractional Big $q$-Laguerre
functions by
\begin{equation}\label{big q-Laguerre}
P_{\lambda}(x;a,b;q)=\dfrac{a^\lambda b^\lambda
q^{\lambda(\lambda+1)}(1-q)^\lambda\left(x,;q\right)_{\infty}}{(aq,bq;q)_\lambda
\left({x\over a},{x\over b};q\right)_{\infty}}
D_q^{\lambda}\left[\dfrac{\left({x\over aq^{\lambda}},{x\over
bq^{\lambda}};q\right)_{\infty}}{\left(x;q\right)_{\infty}} \right].
\end{equation}
\end{definition}

\begin{proposition}
The fractional Big $q$-Laguerre functions defined by relation
(\ref{big q-Laguerre}) have the following basic hypergeometric
representation
\begin{eqnarray*}
P_{\lambda}(x;a,b;q)=\dfrac{a^\lambda b^\lambda
q^{\lambda(\lambda+1)}x^{-\lambda}}{(aq,bq;q)_\lambda \left({x\over
a},{x\over b};q\right)_{-\lambda}}
\qhypergeom{3}{2}{q^{-\lambda},x,0}{{x\over a}q^{-\lambda},{x\over
bc}q^{-\lambda}}{q;q}.
\end{eqnarray*}
\end{proposition}

\begin{proof}
Using the definitions of the fractional Big $q$-Laguerre functions
and  the fractional $q$-derivative (\ref{zeinab}), we have:
\begin{eqnarray*}
P_{\lambda}(x;a,b;q)&=&\dfrac{a^\lambda b^\lambda
q^{\lambda(\lambda+1)}(1-q)^\lambda\left(x;q\right)_{\infty}}{(aq,bq;q)_\lambda
\left({x\over a},{x\over b};q\right)_{\infty}}
D_q^{\lambda}\left[\dfrac{\left({x\over aq^{\lambda}},{x\over
bq^{\lambda}};q\right)_{\infty}}{\left(x;q\right)_{\infty}} \right]\\
&=&\dfrac{a^\lambda b^\lambda
q^{\lambda(\lambda+1)}\left(x;q\right)_{\infty}x^{-\lambda}}{(aq,bq;q)_\lambda
\left({x\over a},{x\over b};q\right)_{\infty}}
\sum_{n=0}^{+\infty}q^n\dfrac{(q^{-\lambda};q)_{n}}{(q;q)_{n}}
\dfrac{\left({xq^{n-\lambda}\over a},{xq^{n-\lambda}\over
b};q\right)_{\infty}}{\left(xq^{n};q\right)_{\infty}}.
\end{eqnarray*}
Using the properties of the $q$-pochhammer symbol, we have:
\begin{eqnarray*}
P_{\lambda}(x;a,b;q)&=&\dfrac{a^\lambda b^\lambda
q^{\lambda(\lambda+1)}\left(x;q\right)_{\infty}x^{-\lambda}}{(aq,bq;q)_\lambda
\left({x\over
a},{x\over b};q\right)_{\infty}}\\
&&\times\sum_{n=0}^{+\infty}q^n\dfrac{(q^{-\lambda};q)_{n}}{(q;q)_{n}}
\dfrac{\left({xq^{-\lambda}\over a},{xq^{-\lambda}\over
b};q\right)_{\infty}}{\left({xq^{-\lambda}\over
a},{xq^{-\lambda}\over b};q\right)_{n}}
\dfrac{\left(x;q\right)_{n}}{\left(x;q\right)_{\infty}}.
\end{eqnarray*}
Using the fact that
$(aq^\lambda;q)_\infty=\dfrac{(a;q)_\infty}{(a;q)_\lambda}$ then, we
have
\begin{eqnarray*}
P_{\lambda}(x;a,b;q)&=&\dfrac{a^\lambda b^\lambda
q^{\lambda(\lambda+1)}x^{-\lambda}}{(aq,bq;q)_\lambda \left({x\over
a},{x\over b};q\right)_{-\lambda}}
\sum_{n=0}^{+\infty}q^n\dfrac{\left(q^{-\lambda},x;q\right)_n}{\left({xq^{-\lambda}\over
a},{xq^{-\lambda}\over
b};q\right)_{n}}\dfrac{q^n}{(q;q)_n}\\
&=&\dfrac{a^\lambda b^\lambda
q^{\lambda(\lambda+1)}x^{-\lambda}}{(aq,bq;q)_\lambda \left({x\over
a},{x\over b};q\right)_{-\lambda}}
\qhypergeom{3}{2}{q^{-\lambda},x,0}{{x\over a}q^{-\lambda},{x\over
b}q^{-\lambda}}{q;q}.
\end{eqnarray*}
\end{proof}

\begin{flushleft}
{\bf 5.4. The fractional Little $q$-Jacobi functions}
\end{flushleft}

\noindent The Little $q$-Jacobi polynomials have the $q$-hypergeometric representation \cite[P.\ 482]{KLS}
\[p_n(x;a,b|q)=\qhypergeom{2}{1}{q^{-n},abq^{n+1}}{aq}{q;qx}.\]
They can also be represented by the Redrigues-type formula
\[w(x;\alpha,\beta;q)p_n(x;q^\alpha,q^\beta|q)=\dfrac{q^{n\alpha+\binom{n}{2}}(1-q)^n}{(q^{\alpha+1};q)_n}D_{q^{-1}}^n[w(x;\alpha+n,\beta+n;q)],\]
where
\[w(x;\alpha,\beta;q)=\dfrac{(qx;q)_\infty}{(q^{\beta+1}x;q)_{\infty}}x^\alpha. \]

\begin{definition}
Let $\lambda\in\R$, we define the fractional Little $q$-Jacobi
functions by
\begin{equation}\label{little q-Jacobi}
P_{\lambda}(x;q^\alpha,q^\beta|q)=\dfrac{q^{{\lambda}\alpha+\textcolor{black}{\frac{\lambda(\lambda-1)}{2}}}(1-q)^{\lambda}
}{w(x;\alpha,\beta;q)(q^{\alpha+1};q)_{\lambda}}D_{q^{-1}}^{\lambda}[w(x;\alpha+{\lambda},\beta+{\lambda};q)].
\end{equation}
\end{definition}

\begin{proposition}
The fractional Little $q$-Jacobi functions defined by relation
(\ref{little q-Jacobi}) have the following basic hypergeometric
representation
\begin{eqnarray*}
P_{\lambda}(x;q^\alpha,q^\beta|q)=(-1)^{\lambda}q^{\lambda \alpha+\frac{\lambda(\lambda+1)}{2}}\dfrac{(q^{-(\alpha+\lambda)};q)_{\infty}}{(q^{\alpha+1};q)_{\lambda}(q^{-\alpha};q)_{\infty}}\qhypergeom{2}{1}{q^{-\lambda},q^{\alpha+\beta+\lambda+1}}{q^{\alpha+1}}{q;qx}.
\end{eqnarray*}
\end{proposition}

\textcolor{black}{
\begin{proof}
Applying the $q$-binomial theorem \eqref{qbinom}, we have  
\begin{eqnarray*}
w(x;\alpha,\beta;q)&=&x^\alpha \dfrac{(qx;q)_\infty}{(q^{\beta+1}x;q)_{\infty}}\\
&=& \sum_{n=0}^{\infty}\dfrac{(q^{-\beta};q)_n}{(q;q)_n}q^{n(\beta+1)}x^{n+\alpha}.
\end{eqnarray*}
Thus, 
\[w(x;\alpha+{\lambda},\beta+{\lambda};q)=\sum_{n=0}^{\infty}\dfrac{(q^{-(\beta+\lambda)};q)_n}{(q;q)_n}q^{n(\beta+\lambda+1)}x^{n+\alpha+\lambda}.\]
Hence, 
\begin{eqnarray*}
&& D_{q^{-1}}^{\lambda}[w(x;\alpha+{\lambda},\beta+{\lambda};q)]\\
&&\hspace*{2cm}=\sum_{n=0}^{\infty}\dfrac{(q^{-(\beta+\lambda)};q)_n}{(q;q)_n}q^{n(\beta+\lambda+1)}D_{q^{-1}}^{\lambda}\left[x^{n+\alpha+\lambda}\right]\\
&&\hspace*{2cm}=\dfrac{q^{\lambda}}{(q-1)^{\lambda}}\sum_{n=0}^{\infty}\dfrac{(q^{-(\beta+\lambda)};q)_n}{(q;q)_n}\dfrac{(q^{-n}q^{-(\alpha+\lambda)};q)_{\infty}}{(q^{-n}q^{-\alpha};q)_{\infty}} q^{n(\beta+\lambda+1)}x^{n+\alpha}.
\end{eqnarray*}
Using the fact that 
\[(aq^{-n};q)_{\infty}=(-a)^n(a^{-1}q;q)_nq^{-\binom{n+1}{2}}(a;q)_{\infty},\]
we have 
\[\dfrac{(q^{-n}q^{-(\alpha+\lambda)};q)_{\infty}}{(q^{-n}q^{-\alpha};q)_{\infty}}=\dfrac{(q^{-(\alpha+\lambda)};q)_{\infty}}{(q^{-\alpha};q)_{\infty}}\dfrac{(q^{\alpha+\lambda+1};q)_n}{(q^{\alpha+1};q)_n}q^{-\lambda n}. \]
Whence 
\begin{eqnarray*}
&&D_{q^{-1}}^{\lambda}[w(x;\alpha+{\lambda},\beta+{\lambda};q)]\\
&&\hspace*{1cm}=\dfrac{q^{\lambda}x^{\alpha}}{(q-1)^{\lambda}}\dfrac{(q^{-(\alpha+\lambda)};q)_{\infty}}{(q^{-\alpha};q)_{\infty}}\sum_{n=0}^{\infty}\dfrac{(q^{-(\beta+\lambda)};q)_n}{(q;q)_n}\dfrac{(q^{\alpha+\lambda+1};q)_n}{(q^{\alpha+1};q)_n} q^{n(\beta+1)}x^{n}\\
&&\hspace*{1cm}=\dfrac{q^{\lambda}x^{\alpha}}{(q-1)^{\lambda}}\dfrac{(q^{-(\alpha+\lambda)};q)_{\infty}}{(q^{-\alpha};q)_{\infty}}\qhypergeom{2}{1}{q^{-(\beta+\lambda)},q^{\alpha+\lambda+1}}{q^{\alpha+1}}{q;xq^{\beta+1}}
\end{eqnarray*}
Now, applying the Heine-Euler transformation formula \eqref{heine2}, we have 
\[\qhypergeom{2}{1}{q^{-(\beta+\lambda)},q^{\alpha+\lambda+1}}{q^{\alpha+1}}{q;xq^{\beta+1}}=\dfrac{(qx;q)_{\infty}}{(q^{\beta+1};q)_{\infty}}\qhypergeom{2}{1}{q^{-\lambda},q^{\alpha+\beta+\lambda+1}}{q^{\alpha+1}}{q;qx}.\]
So we have 
\[\dfrac{D_{q^{-1}}^{\lambda}[w(x;\alpha+{\lambda},\beta+{\lambda};q)]}{w(x;\alpha,\beta;q)}=\dfrac{q^{\lambda}}{(q-1)^{\lambda}}\dfrac{(q^{-(\alpha+\lambda)};q)_{\infty}}{(q^{-\alpha};q)_{\infty}}\qhypergeom{2}{1}{q^{-\lambda},q^{\alpha+\beta+\lambda+1}}{q^{\alpha+1}}{q;qx}.\]
We finally obtain 
\[P_\lambda(x;q^{\alpha},q^{\beta};q)=(-1)^{\lambda}q^{\lambda \alpha+\frac{\lambda(\lambda+1)}{2}}\dfrac{(q^{-(\alpha+\lambda)};q)_{\infty}}{(q^{\alpha+1};q)_{\lambda}(q^{-\alpha};q)_{\infty}}\qhypergeom{2}{1}{q^{-\lambda},q^{\alpha+\beta+\lambda+1}}{q^{\alpha+1}}{q;qx}.\]
\end{proof}}

\begin{proposition} The following limit transition holds:
\[  \lim\limits_{\lambda\to n}P_\lambda(x;q^{\alpha},q^{\beta};q)=p_n(x;q^{\alpha},q^{\beta};q),\]
where $n$ is a nonnegative integer. 
\end{proposition}

\begin{proof}
It is not difficult to see that
\begin{eqnarray*}
&&\lim\limits_{\lambda\to n}\left((-1)^{\lambda}q^{\lambda \alpha+\frac{\lambda(\lambda+1)}{2}}\dfrac{(q^{-(\alpha+\lambda)};q)_{\infty}}{(q^{\alpha+1};q)_{\lambda}(q^{-\alpha};q)_{\infty}}\right)\\
&&\hspace*{3cm}=(-1)^{n}q^{n \alpha+\binom{n+1}{2}}\dfrac{(q^{-(\alpha+n)};q)_{\infty}}{(q^{\alpha+1};q)_{n}(q^{-\alpha};q)_{\infty}}\\
&&\hspace*{3cm}=(-1)^{n}q^{n \alpha+\binom{n+1}{2}}\dfrac{(-1)^nq^{-\alpha n}q^{-\binom{n+1}{2}}(q^{\alpha+1};q)_{n}(q^{-\alpha};q)_{\infty}}{(q^{\alpha+1};q)_{n}(q^{-\alpha};q)_{\infty}}=1,
\end{eqnarray*}
so, 
\[\lim\limits_{\lambda\to n}P_{\lambda}(x;q^{\alpha};q)=p_n(x;q^{\alpha};q).\]
\end{proof}

\begin{flushleft}
{\bf 5.5. The fractional Little $q$-Laguerre functions}
\end{flushleft}

\noindent The Little $q$-Laguerre polynomials have the $q$-hypergeometric representation \cite[P.\ 518]{KLS}
\[p_n(x,a|q)= \qhypergeom{2}{1}{q^{-n},0}{aq}{q;qx}.\]
They can also be represented by the Rodrigues-type formula \cite[P.\ 520]{KLS}
\[w(x;\alpha;q)p_n(x;q^\alpha;q)=\dfrac{q^{n\alpha+\binom{n}{2}}(1-q)^n}{(q^{\alpha+1};q)_n}D_{q^{-1}}^n[w(x;\alpha+n;q)],\]
with \[w(x;\alpha;q)=x^\alpha(qx;q)_{\infty}.\]

\begin{definition}
Let $\lambda\in\R$, we define the fractional Little $q$-Laguerre
functions by
\begin{equation}\label{little q-Laguerre}
P_{\lambda}(x;q^\alpha;q)=\dfrac{q^{{\lambda}\alpha+\binom{{\lambda}}{2}}(1-q)^{\lambda}
}{w(x;\alpha;q)(q^{\alpha+1};q)_{\lambda}}D_{q^{-1}}^{\lambda}[w(x;\alpha+{\lambda};q)].
\end{equation}
\end{definition}

\begin{proposition}
The fractional Little $q$-Laguerre functions defined by relation
(\ref{little q-Laguerre}) have the following representation
\begin{eqnarray*}
P_{\lambda}(x;q^\alpha|q)=(-1)^{\lambda}q^{\lambda \alpha+\frac{\lambda(\lambda+1)}{2}}\dfrac{(q^{-(\alpha+\lambda)};q)_{\infty}}{(q^{\alpha+1};q)_{\lambda}(q^{-\alpha};q)_{\infty}}\qhypergeom{2}{1}{q^{-\lambda},0}{q^{\alpha+1}}{qx}.
\end{eqnarray*}
\end{proposition}

\textcolor{black}{
\begin{proof}
It is easy to see that 
\[w(x;\alpha;q)=x^{\alpha}(qx;q)_{\infty}=\sum_{n=0}^{\infty}\dfrac{q^{\binom{n+1}{2}}}{(q;q)_n}(-1)^nx^{n+\alpha}.\]
So 
\[w(x;\alpha+\lambda;q)=\sum_{n=0}^{\infty}\dfrac{q^{\binom{n+1}{2}}}{(q;q)_n}(-1)^nx^{n+\alpha+\lambda}.\]
Thus using the definitions of the fractional Little $q$-Laguerre
functions and the fractional $q$-derivative (\ref{def-dq-1}), combined with the transformations \eqref{heine} and \eqref{heine3} we
have:
\begin{eqnarray*}
D_{q^{-1}}^{\lambda}[w(x;\alpha+{\lambda};q)]&=& \sum_{n=0}^{\infty}\dfrac{q^{\binom{n+1}{2}}}{(q;q)_n}(-1)^nD_{q^{-1}}^{\lambda}[x^{n+\alpha+\lambda}]\\
&=&\sum_{n=0}^{\infty}\dfrac{q^{\binom{n+1}{2}}}{(q;q)_n}(-1)^n\dfrac{q^{\lambda}}{(q-1)^{\lambda}}\dfrac{(q^{-n}q^{-(\alpha+\lambda)};q)_{\infty}}{(q^{-n}q^{-\alpha};q)_{\infty}}x^{n+\alpha}\\
&=&\dfrac{x^\alpha q^{\lambda}}{(q-1)^{\lambda}}\dfrac{(q^{-(\alpha+\lambda)};q)_{\infty}}{(q^{-\alpha};q)_{\infty}}\sum_{n=0}^{\infty}\dfrac{(q^{\alpha+\lambda+1};q)_n}{(q^{\alpha+1};q)_n}\left((-1)^nq^{\binom{n}{2}}\right)\dfrac{(q^{1-\lambda} x)^n}{(q;q)_n}\\
&=&\dfrac{x^\alpha q^{\lambda}}{(q-1)^{\lambda}}\dfrac{(q^{-(\alpha+\lambda)};q)_{\infty}}{(q^{-\alpha};q)_{\infty}}\qhypergeom{1}{1}{q^{\alpha+\lambda+1}}{q^{\alpha+1}}{q;xq^{-\lambda+1}}\\
& \overset{\eqref{heine3}}{=}&\dfrac{x^\alpha q^{\lambda}}{(q-1)^{\lambda}}\dfrac{(q^{-(\alpha+\lambda)};q)_{\infty}}{(q^{-\alpha};q)_{\infty}}
\dfrac{(q^{-\lambda};q)_{\infty}}{(q^{\alpha+1};q)_{\infty}}\qhypergeom{2}{0}{q^{\alpha+\lambda+1},qx}{0}{q^{-\lambda}}\\
&\overset{\eqref{heine}}{=}&\dfrac{x^\alpha (qx;q)_{\infty} q^{\lambda}}{(q-1)^{\lambda}}\dfrac{(q^{-(\alpha+\lambda)};q)_{\infty}}{(q^{-\alpha};q)_{\infty}}
\qhypergeom{2}{1}{q^{-\lambda},0}{q^{\alpha+1}}{qx}\\
&=&(-1)^{\lambda}q^{\lambda \alpha+\frac{\lambda(\lambda+1)}{2}}\dfrac{(q^{-(\alpha+\lambda)};q)_{\infty}}{(q^{\alpha+1};q)_{\lambda}(q^{-\alpha};q)_{\infty}}\qhypergeom{2}{1}{q^{-\lambda},0}{q^{\alpha+1}}{qx}.
\end{eqnarray*}\end{proof}}

\begin{proposition}
Let $n$ be a nonnegative integer, then the following limit transition holds true 
\[\lim\limits_{\lambda\to n}P_{\lambda}(x;q^{\alpha};q)=p_n(x;q^{\alpha};q).\]
\end{proposition}

\begin{proof} 
Since 
\begin{eqnarray*}
&&\lim\limits_{\lambda\to n}\left((-1)^{\lambda}q^{\lambda \alpha+\frac{\lambda(\lambda+1)}{2}}\dfrac{(q^{-(\alpha+\lambda)};q)_{\infty}}{(q^{\alpha+1};q)_{\lambda}(q^{-\alpha};q)_{\infty}}\right)\\
&&\hspace*{3cm}=(-1)^{n}q^{n \alpha+\binom{n+1}{2}}\dfrac{(q^{-(\alpha+n)};q)_{\infty}}{(q^{\alpha+1};q)_{n}(q^{-\alpha};q)_{\infty}}\\
&&\hspace*{3cm}=(-1)^{n}q^{n \alpha+\binom{n+1}{2}}\dfrac{(-1)^nq^{-\alpha n}q^{-\binom{n+1}{2}}(q^{\alpha+1};q)_{n}(q^{-\alpha};q)_{\infty}}{(q^{\alpha+1};q)_{n}(q^{-\alpha};q)_{\infty}}=1,
\end{eqnarray*}
we have, $\displaystyle{\lim\limits_{\lambda\to n}P_{\lambda}(x;q^{\alpha};q)=p_n(x;q^{\alpha};q)}$.
\end{proof}

%

\begin{flushleft}
{\bf 5.6. The fractional $q$-Laguerre functions}
\end{flushleft}

\noindent The $q$-Laguerre polynomials have the $q$-hypergeometric representation \cite[P.\ 522]{KLS}
\[L_n^{(\alpha)}(x)=\frac{(q^{\alpha+1};q)_n}{(q;q)_n} \qhypergeom{1}{1}{q^{-n}}{q^{\alpha+1}}{q;-q^{n+\alpha+1}x}.\]
They can also be represented by the Rodrigues-type formula \cite[P.\ 524]{KLS}
\[w(x;\alpha;q)L_n^{(\alpha)}(x;q)=\dfrac{(1-q)^n}{(q;q)_n}D_q^n[w(x;\alpha+n;q)],\]
with
\[w(x;\alpha;q)=\dfrac{x^\alpha}{(-x;q)_{\infty}}.\]

\begin{definition}
Let $\lambda\in\R$, we define the fractional  $q$-Laguerre functions
by
\begin{equation}\label{q-Laguerre}
L_{\lambda}^{\alpha}(x;q)=\dfrac{
(1-q)^\lambda\left(-x,;q\right)_{\infty}}{(q;q)_\lambda x^\alpha}
D_q^{\lambda}\left[\dfrac{x^{\alpha+\lambda}}{\left(-x;q\right)_{\infty}}
\right].
\end{equation}
\end{definition}

\begin{proposition}
The fractional  $q$-Laguerre functions defined by relation
(\ref{q-Laguerre}) have the following basic hypergeometric
representation
\begin{eqnarray*}
L_{\lambda}^{\alpha}(x;q)=\dfrac{ 1}{(q;q)_\lambda }
\qhypergeom{2}{0}{q^{-\lambda},-x}{-}{q;q^{\alpha+\lambda+1}}.
\end{eqnarray*}
\end{proposition}


\begin{proof}
For the definition of the $q$-exponential \eqref{small-qexp}, it follows that 
\[\dfrac{x^{\alpha+\lambda}}{(-x;q)_{\infty}}=x^{\alpha+\lambda}e_q(-x)=\sum_{n=0}^{\infty}\dfrac{(-1)^nx^{\alpha+\lambda+n}}{(q;q)_n}.\]
Then, 
\begin{eqnarray*}
 D_{q,0}^{\lambda}\left[\dfrac{x^{\alpha+\lambda}}{(-x;q)_{\infty}} \right]&=&\sum_{n=0}^{\infty}\dfrac{(-1)^n}{(q;q)_n}\dfrac{\Gamma_q(\alpha+\lambda+n+1)}{\Gamma_q(\alpha+n+1)}x^{\alpha+n}\\
 &=&x^{\alpha}\dfrac{\Gamma_q(\alpha+\lambda+1)}{\Gamma_q(\alpha+1)}\sum_{n=0}^{\infty}\dfrac{(q^{\alpha+\lambda+1};q)_n}{(q;q)_n(q^{\alpha+1};q)_n}(-x)^n\\
 &=&x^{\alpha}\dfrac{\Gamma_q(\alpha+\lambda+1)}{\Gamma_q(\alpha+1)}\qhypergeom{2}{1}{0,q^{\alpha+\lambda+1}}{q^{\alpha+1}}{q;-x}\\
 &=& \dfrac{x^{\alpha}}{(1-q)^{\lambda}}\dfrac{(q^{\alpha+1};q)_{\infty}}{(q^{\alpha+\lambda+1};q)_{\infty}}\qhypergeom{2}{1}{0,q^{\alpha+\lambda+1}}{q^{\alpha+1}}{q;-x}.
\end{eqnarray*}
Using the Heine transformation \eqref{heine}, we have 
\[\qhypergeom{2}{1}{0,q^{\alpha+\lambda+1}}{q^{\alpha+1}}{q;-x}=\dfrac{(q^{\alpha+\lambda+1};q)_{\infty}}{(q^{\alpha+1};q)_{\infty}(-x;q)_{\infty}}\qhypergeom{2}{1}{q^{-\lambda};-x}{0}{q;q^{\alpha+\lambda+1}}.\]
Hence, we have 
\[D_{q,0}^{\lambda}\left[\dfrac{x^{\alpha+\lambda}}{(-x;q)_{\infty}} \right]=\dfrac{x^{\alpha}}{(1-q)^{\lambda}(-x;q)_{\infty}}\qhypergeom{2}{1}{q^{-\lambda};-x}{0}{q;q^{\alpha+\lambda+1}}.\]
Finally it follows that 
\[L_{\lambda}^{(\alpha)}(x;q)=\dfrac{1}{(q;q)_{\lambda}}\qhypergeom{2}{1}{q^{-\lambda};-x}{0}{q;q^{\alpha+\lambda+1}}.\]
\end{proof}

\begin{flushleft}
{\bf 5.7. The fractional Al Salam-Carlitz I functions}
\end{flushleft}

\noindent The Al Salam-Carlitz I polynomials have the hypergeometric representation \cite[P.\ 534]{KLS}
\begin{eqnarray*}
U_n^{(a)}(x;q)&=&(-a)^nq^{n\choose2}\qhypergeom{2}{1}{q^{-n},x^{-1}}{0}{q;\frac{qx}{a}}.
\end{eqnarray*}
They can also be represented by the Rodrigues-type formula \cite[P.\ 536]{KLS}
\[w(x;a;q)U_n^{(a)}(x;q)=a^nq^{\frac{1}{2}n(n-3)}(1-q)^nD_{q^{-1}}^n[w(x;a;q)],\]
with
\[w(x;a;q)=(qx,a^{-1}qx;q)_{\infty}.\]

\begin{definition}
Let $\lambda\in\R$, we define the fractional Al Salam-Carlitz I
functions by
\begin{equation}\label{al salam carlitz I}
U^{(a)}_{\lambda}(x;q)=\dfrac{a^\lambda
q^{{{\lambda}(\lambda-3)\over 2}}(1-q)^{\lambda}
}{w(x;a;q)}D_{q^{-1}}^{\lambda}[w(x;a;q)].
\end{equation}
\end{definition}

\begin{proposition}
The fractional Little Al Salam-Carlitz I functions defined by
relation (\ref{al salam carlitz I}) have the following
representation
\begin{eqnarray*}
U^{(a)}_{\lambda}(x;q)=(-a)^\lambda  q^{\frac{\lambda(\lambda-1)} {2}}x^{-\lambda}\qhypergeom{3}{0}{q^{-\lambda},x^{-1},ax^{-1}}{-}{q;-\frac{q^{\lambda}x^2}{a}}.
\end{eqnarray*}
\end{proposition}

\begin{proof}
Using the definitions of the fractional Al Salam-Carlitz I functions
and the fractional $q$-derivative (\ref{def-dq-1}), we have:
\begin{eqnarray*}
U^{(a)}_{\lambda}(x;q)&=&\dfrac{a^\lambda
q^{{{\lambda}(\lambda-3)\over 2}}(1-q)^{\lambda}
}{w(x;a;q)}D_{q^{-1}}^{\lambda}[w(x;a;q)]\\
&=&\dfrac{a^\lambda q^{{{\lambda}(\lambda-3)\over
2}}(1-q)^{\lambda}}{ (qx,a^{-1}qx ;q)_{\infty} }
D_{q^{-1}}^{\lambda}\left[(qx, a^{-1}qx ;q)_{\infty}\right]\\
&=&\dfrac{a^\lambda q^{{{\lambda}(\lambda-3)\over
2}}(1-q)^{\lambda}}{ (qx,a^{-1}qx ;q)_{\infty}
}\dfrac{q^\lambda}{(q-1)^\lambda x^\lambda}
\sum_{n=0}^{+\infty}q^{\lambda
n}\dfrac{(q^{-\lambda};q)_{n}(qxq^{-n},a^{-1}qxq^{-n};q)_{\infty}}{(q;q)_{n}}\\
&=&\dfrac{a^\lambda q^{{{\lambda}(\lambda-3)\over
2}}(-1)^{\lambda}}{ (qx,a^{-1}qx ;q)_{\infty} }\dfrac{q^\lambda}{
x^\lambda} \sum_{n=0}^{+\infty}q^{\lambda
n}\dfrac{(q^{-\lambda};q)_{n}(qxq^{-n},a^{-1}qxq^{-n};q)_{\infty}}{(q;q)_{n}}.
\end{eqnarray*}

\noindent Using the fact that
$(aq^{-n};q)_\infty=a^n(-1)^n\left(a^{-1}q;q\right)_nq^{-\binom{n+1}{2}}(a;q)_\infty$
 then, we have
\begin{eqnarray*}
U^{(a)}_{\lambda}(x;q)&=&(-a)^\lambda \left({q\over
x}\right)^\lambda q^{{{\lambda}(\lambda-3)\over 2}}
\sum_{n=0}^{+\infty}\dfrac{(q^{-\lambda};q)_{n}}{(q;q)_{n}}
\dfrac{\left(x^{-1};q\right)_n \left(ax^{-1};q\right)_n({\frac{x^2q^{\lambda}}{
a}})^n}{q^{2\binom{n}{2}}}\\
&=&\textcolor{black}{(-a)^\lambda  q^{\frac{\lambda(\lambda-3)} {2}}\left({\frac{q}{
x}}\right)^\lambda\qhypergeom{3}{0}{q^{-\lambda},x^{-1},ax^{-1}}{-}{q;-\frac{q^{\lambda}x^2}{a}}.}
\end{eqnarray*}
\end{proof}

\begin{flushleft}
{\bf 5.8. The fractional Al Salam-Carlitz II functions}
\end{flushleft}

\noindent The Al Salam-Carlitz II polynomials have the hypergeometric representation \cite[P.\ 537]{KLS}
\begin{eqnarray*}
V_n^{(a)}(x;q)&=&(-a)^nq^{-{n\choose2}}\qhypergeom{2}{0}{q^{-n},x}{0}{q;\frac{q^n}{a}}.
\end{eqnarray*}
They can also be represented by the Rodrigues-type formula  \cite[P.\ 539]{KLS}
\[w(x;a;q)V_n^{(a)}(x;q)=a^n(q-1)^nq^{-\binom{n}{2}}D_q^n[w(x;a;q)],\]
with
\[ w(x;a;q)=\dfrac{1}{(x;a^{-1}x;q)_{\infty}}.\]

\begin{definition}
Let $\lambda\in\R$, we define the fractional  Al Salam-Carlitz II
functions by
\begin{equation}\label{al salam-carlitz}
V_{\lambda}^{(a)}(x;q)=a^\lambda q^{-\frac{\lambda(\lambda-1)}{2}}(q-1)^\lambda\left(x,a^{-1}x;q\right)_{\infty}
D_q^{\lambda}\left[\dfrac{1}{\left(x,a^{-1}x;q\right)_{\infty}}
\right].
\end{equation}
\end{definition}

\begin{proposition}
The fractional Al Salam-Carlitz II functions defined by relation
(\ref{al salam-carlitz}) have the following basic hypergeometric
representation
\textcolor{black}{\begin{eqnarray*}
V_{\lambda}^{(a)}(x;q)= q^{-\frac{\lambda(\lambda-1)}{2}}(-1)^{\lambda} \left(\frac{a}{x}\right)^{\lambda}
\qhypergeom{3}{2}{q^{-\lambda},x,{xa^{-1}}}{0,0}{q;q}.
\end{eqnarray*}}
\end{proposition}

\begin{flushleft}
{\bf 5.9. The fractional Stieltjes-Wigert functions}
\end{flushleft}

\noindent The Stieltjes-Wigert polynomials have the $q$-hypergeometric representation \cite[P.\ 544]{KLS}
\begin{eqnarray*}
S_n(x;q)&=&\frac{1}{(q;q)_n}\qhypergeom{1}{1}{q^{-n}}{0}{q;-q^{n+1}x}.
\end{eqnarray*}
They can also be represented by the Rodrigues-type formula \cite[P.\ 545]{KLS}
\[w(x;q)S_n(x;q)=\dfrac{q^n(1-q)^n}{(q;q)_n}D_q^n[w(q^nx;q)],\]
with
\[w(x;q)=\dfrac{1}{(-x,-qx^{-1};q)_{\infty}}.\]

\begin{definition}
Let $\lambda\in\R$, we define the fractional Stieltjes-Wigert  
functions by
\begin{equation}\label{stiljes-wigert}
w(x;q)S_\lambda(x;q)=\dfrac{q^{\lambda}(1-q)^{\lambda}}{(q;q)_\lambda}D_q^{\lambda}[w(q^\lambda x;q)],
\end{equation}
with
\[w(x;q)=\dfrac{1}{(-x,-qx^{-1};q)_{\infty}}.\]
\end{definition}

\begin{proposition}
The fractional Stieltjes-Wigert functions defined by relation
\eqref{stiljes-wigert} have the following basic hypergeometric
representation
\begin{equation}
S_{\lambda}(x;q)=\dfrac{(-x;q)_{\lambda}(-qx^{-1};q)_{-\lambda}}{(q;q)_{\lambda}}\left(\dfrac{q}{x}\right)^{\lambda}\qhypergeom{1}{1}{q^{-\lambda}}{0}{q;-xq^{\lambda+1}}.
\end{equation}
\end{proposition}

\begin{proof}
Using the definitions of the fractional Stieltjes-Wigert functions
and the fractional $q$-derivative (\ref{def-dq-1}), we have:
\begin{eqnarray*}
S_{\lambda}(x;q)&=&\dfrac{1}{(q;q)_{\lambda}}\left(\dfrac{q}{x}\right)^{\lambda}\dfrac{1}{w(x;q)}\sum_{n=0}^{\infty}\dfrac{(q^{-\lambda};q)_n}{(q;q)_n}q^nw(q^{\lambda+n}x;q).
\end{eqnarray*}
It is not difficult to see that 
\[w(q^{\lambda+n}x;q)=(q^{\lambda}x)^nq^{\binom{n}{2}}w(q^{\lambda}x;q).\]
Hence 
\begin{eqnarray*}
S_{\lambda}(x;q)&=&\dfrac{1}{(q;q)_{\lambda}}\left(\dfrac{q}{x}\right)^{\lambda}\dfrac{w(q^{\lambda}x;q)}{w(x;q)}\sum_{n=0}^{\infty}\dfrac{(q^{-\lambda};q)_n}{(q;q)_n}(q^{\lambda+1}x)^nq^{\binom{n}{2}}\\
&=&\dfrac{1}{(q;q)_{\lambda}}\left(\dfrac{q}{x}\right)^{\lambda}\dfrac{w(q^{\lambda}x;q)}{w(x;q)}\qhypergeom{1}{1}{q^{-\lambda}}{0}{q;-xq^{\lambda+1}}.
\end{eqnarray*}
The result follows by a bit of simplification. 
\end{proof}

\begin{proposition}
Let $n$ be a nonnegative integer. Then the following transition limit holds.
\begin{equation}
 \lim\limits_{\lambda\to n}S_{\lambda}(x;q)=q^{\binom{n+1}{2}}S_{n}(x;q).  
\end{equation}
\end{proposition}

\vspace{1.5cc}
\begin{center}
{\bf 6. FRACTIONAL $q$-GAUSS DIFFERENTIAL EQUATION}
\end{center}

In this section, we give a fractional version of the $q$-hypergeometric $q$-difference equation given by Koorwinder in \cite{koorwinder}. Next, we solve this fractional $q$-difference equation by means of modified power series.

\begin{definition}
The fractional $q$-hypergeometric $q$-difference equation is defined
for $0<\lambda\leq 1$ by
\begin{eqnarray}\label{frac-q-diff-eqn}
&&z^\lambda(q^c-q^{a+b+1}z^\lambda)(D^{2\lambda}_q u)(z) \label{korn2} \\
&&\qquad\qquad
+\left([c]_q-(q^b[a]_q+q^a[b+1]_q)z^\lambda\right)(D_q^\lambda
u)(z)-[a]_q[b]_qu(z)=0.\nonumber
\end{eqnarray}
\end{definition}

\begin{definition}
The fractional $q$-Gauss function is defined as the series
\begin{equation}\label{frac-q-gauss}
\qhypergeomu{2}{1}{q^a,q^b}{q^c}{q;z}=u_0
z^{\rho}\sum_{n=0}^{\infty}\prod_{k=0}^{n}\dfrac{g_{q,k}(\rho)}{f_{q,k+1}(\rho)}z^{n\lambda},\quad
0<\lambda\leq 1,
\end{equation}
where
\begin{equation}\label{fq}
f_{q,k}(\rho)=q^c\dfrac{\Gamma_q(1+\rho+k\lambda)}{\Gamma_q(1+\rho+(k-2)\lambda)}+[c]_q\dfrac{\Gamma_q(1+\rho+k\lambda)}{\Gamma_q(1+\rho+(k-1)\lambda)},
\end{equation}
\begin{eqnarray}
g_{q,k}(\rho)&=&q^{a+b+1}\dfrac{\Gamma_q(1+\rho+k\lambda)}{\Gamma_q(1+\rho+(k-2)\lambda)}\label{gq}\\
\nonumber&&+(q^b[a]_q+q^a[b+1]_q)\dfrac{\Gamma_q(1+\rho+k\lambda)}{\Gamma_q(1+\rho+(k-1)\lambda)}+[a]_q[b]_q,
\end{eqnarray}
and $\rho>-1$ satisfies the equation
\begin{equation}
    f_{q,0}(\rho)=\dfrac{\Gamma_q(1+\rho)}{\Gamma_q(1+\rho-2\lambda)}+[c]_q\dfrac{\Gamma_q(1+\rho)}{\Gamma_q(1+\rho-\lambda)}=0.
\end{equation}
\end{definition}

The following assertion is valid.

\begin{theorem}
The fractions $q$-Gauss function \eqref{frac-q-gauss} is a solution of the fractonal $q$-Gauss hypergeometric equation \eqref{frac-q-diff-eqn}.
\end{theorem}

\begin{proof}
We look for the solution under  the following modified formal power series form
\[u(z)=\sum_{n=0}^{\infty}d_nz^{n\lambda+\rho}.\]
Then,
\[D_q^{\lambda}u(z)=\sum_{n=0}^{\infty}d_n\dfrac{\Gamma_q(n\lambda+\rho+1)}{\Gamma_q((n-1)\lambda+\rho+1)}z^{(n-1)\lambda+\rho},\]
and
\[D_q^{2\lambda}u(z)=\sum_{n=0}^{\infty}d_n\dfrac{\Gamma_q(n\lambda+\rho+1)}{\Gamma_q((n-2)\lambda+\rho+1)}z^{(n-2)\lambda+\rho}.\]
Inserting these fractional $q$-derivatives in \eqref{frac-q-diff-eqn}, we obtain the following recurrence relation for the coefficients $a_n$,
\[f_{q,n}(\rho)d_{n+1}-g_{q,n+1}(\rho)d_{n}=0, \]
with $f_{q,0}(\rho)=0$, where $f_{q,n}(\rho)$ and $g_{q,n}(\rho)$ are given by \eqref{fq} and \eqref{gq} respectively. The theorem follows easily.
\end{proof}

\vspace{1.5cc}
\begin{center}
{\bf ACKNOWLEDGEMENT}
\end{center}

\noindent The first  author would like to thank TWAS and DFG for their support to sponsor a research visit at the Institute of Mathematics of the University of Kassel in 2015 under the reference number 3240278140, where part of this work has been done.

%
%

\vspace{1cc}


{\small
\noindent

}\end{document}